\documentclass{article}
\usepackage{amsmath}
\usepackage{amssymb}
\usepackage{amsthm}
\usepackage[margin=1in]{geometry}
\usepackage[utf8]{inputenc}
\usepackage{tikz}
\usetikzlibrary{arrows}

\newtheorem{theorem}{Theorem}
\newtheorem{corollary}{Corollary}[theorem]
\newtheorem{lemma}[theorem]{Lemma}
\newtheorem*{conjecture}{Conjecture}

\title{The Game of Cycles for Grids and Select Theta Graphs}
\author{Christopher Barua, Eric Burkholder, Gabriel Fragoso, Zsuzsanna Szaniszlo}

\begin{document}

\maketitle

\section*{Abstract}

We are investigating who has the winning strategy in a game in which two players take turns drawing arrows trying to complete cycle cells in a graph. A cycle cell is a cycle with no chords. We examine gameboards where the winning strategy was previously unknown. We analyze the game for Theta graphs, generalized Theta graphs with a path length restriction, and grids. Theta graphs have two degree three vertices, and every other vertex has degree two. $\Theta_{k,l,m} $ is a graph with $k+l+m-1$ vertices where the two degree $3$ vertices are connected by disjoint paths of length $k,l$ and $m$. Generalized Theta graphs have two higher degree vertices of the same degree and every other vertex has degree two. We answer the open question related to $\Theta_{3,2,5}$ from the article \textit{The Game of Cycles} \cite{alvarado2021game}, and offer some new directions of research related to this game. 

\section{Introduction}

The Game of Cycles, introduced by Dr. Francis Su in \textit{Mathematics for Human Flourishing} \cite{su2020mathematics} is a two-player game on a graph. The \textit{gameboard} is a simple, connected, undirected, planar graph with its embedding in the plane. The two players take turns directing edges of the graph using arrows, and the goal of the game is to be either the last player to make a move, or be the first player to create a \textit{cycle cell}, which is a chordless cycle of the graph whose edges are all directed the same way. Figure \ref{fig:cycle-cell-defintion} provides an example of how a cycle is not the same as a cycle cell, as the image on the right has a directed cycle, but not a cycle cell.

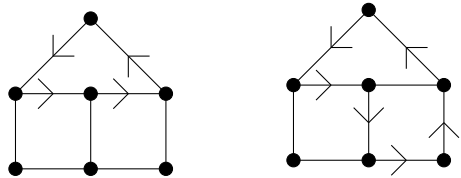
\begin{figure}
    \centering
    \begin{tikzpicture}
    \draw (-1,0)--(1,0);
    \draw (-1,0)--(0,1);
    \draw (1,0)--(0,1);
    \draw (0.5,0)--(0.3,0.2);
    \draw (0.5,0)--(0.3,-0.2);
    \draw (-0.5,0)--(-0.7,0.2);
    \draw (-0.5,0)--(-0.7,-0.2);
    \draw (0.5,0.5)--(0.7828,0.5);
    \draw (0.5,0.5)--(0.5,0.2172);
    \draw (-0.5,0.5)--(-0.5,0.7828);
    \draw (-0.5,0.5)--(-0.2172,0.5);
    \draw (-1,0)--(-1,-1);
    \draw (1,0)--(1,-1);
    \draw (-1,-1)--(1,-1);
    \draw (0,0)--(0,-1);
    \draw [fill=black] (0,0) circle (2.5pt);
    \draw [fill=black] (1,0) circle (2.5pt);
    \draw [fill=black] (-1,0) circle (2.5pt);
    \draw [fill=black] (0,1) circle (2.5pt);
    \draw [fill=black] (-1,-1) circle (2.5pt);
    \draw [fill=black] (1,-1) circle (2.5pt);
    \draw [fill=black] (0,-1) circle (2.5pt);
    \end{tikzpicture}
    \hspace{0.5in}
    \begin{tikzpicture}
    \draw (-1,0)--(1,0);
    \draw (-1,0)--(0,1);
    \draw (1,0)--(0,1);
    \draw (-0.5,0)--(-0.7,0.2);
    \draw (-0.5,0)--(-0.7,-0.2);
    \draw (0.5,0.5)--(0.7828,0.5);
    \draw (0.5,0.5)--(0.5,0.2172);
    \draw (-0.5,0.5)--(-0.5,0.7828);
    \draw (-0.5,0.5)--(-0.2172,0.5);
    \draw (0,-0.5)--(-0.2,-0.3);
    \draw (0,-0.5)--(0.2,-0.3);
    \draw (0.5,-1)--(0.3,-0.8);
    \draw (0.5,-1)--(0.3,-1.2);
    \draw (1,-0.5)--(0.8,-0.7);
    \draw (1,-0.5)--(1.2,-0.7);
    \draw (-1,0)--(-1,-1);
    \draw (1,0)--(1,-1);
    \draw (-1,-1)--(1,-1);
    \draw (0,0)--(0,-1);
    \draw [fill=black] (0,0) circle (2.5pt);
    \draw [fill=black] (1,0) circle (2.5pt);
    \draw [fill=black] (-1,0) circle (2.5pt);
    \draw [fill=black] (0,1) circle (2.5pt);
    \draw [fill=black] (-1,-1) circle (2.5pt);
    \draw [fill=black] (1,-1) circle (2.5pt);
    \draw [fill=black] (0,-1) circle (2.5pt);
    \end{tikzpicture}
    \caption{Two partially filled-in gameboards, the left with a cycle cell, and the right without a cycle cell}
    \label{fig:cycle-cell-defintion}
\end{figure}

\bigskip

\noindent The \textit{rules} of the game are as follows:

\begin{itemize}
    \item A player must make a move if it is possible
    \item Neither player is allowed to create a \textit{sink} or a \textit{source}, which are vertices in which all edges incident to the vertex are directed inwards or outwards, respectively.
\end{itemize}

\noindent The \textit{winner} of the game is the player who first creates a directed cycle cell, or the player who makes the last available move.

\bigskip
\noindent In this paper we analyze the game where the gameboards are Theta graphs and generalized Theta graphs. We also provide the analysis of the previously unknown cases when the gameboards are rectangular grids.

\bigskip

A \textit{Theta graph} is a graph with exactly two vertices of degree three, and with every other vertex of degree two.

\bigskip
\noindent $\Theta_{j,k,l}$ is a graph where 2 vertices are connected with disjoint paths of length $j,k,$ and $l$. Throughout this paper we will concentrate on the number of edges on these paths.

\bigskip
\noindent 
Our main results are as follows:

\bigskip
\noindent
For generalized Theta graphs we find

\bigskip
\textbf{Corollary} \ref{cor:p-paths-extended}
Consider a generalized Theta graph consisting of $p$ disjoint paths between two vertices $u$ and $v$ with the property that every path with an odd number of edges between $u$ and $v$ has at least 5 edges, except for graphs that have an odd number of total edges, there may be a single path of length 3 or a single path of length 1. For such a graph, Player 1 has a winning strategy when there is an odd number of total edges and Player 2 has a winning strategy when there is an even number of total edges.

\bigskip
\textbf{Corollary} \ref{cor:j-2-k-full}
Consider the graph $\Theta_{j,2,k}$ with $j \geq 1, k \geq 2$. If $j+k$ is odd, Player 1 has a winning strategy. If $j+k$ is even, Player 2 has a winning strategy.

\bigskip
\noindent For grids we prove the following Theorem:

\bigskip
\textbf{Theorem} \ref{thm:grids}. Consider a graph that is a grid of $a \times b$ vertices, where exactly one of $a$ or $b$ is odd. Player 1 has a winning strategy for this gameboard.

\bigskip
\noindent Cases for other types of rectangular grids were settled in \cite{alvarado2021game}.

\section{A Sample Game}

To understand the game better we will demonstrate the analysis of the game on a small gameboard. When a gameboard is completed with no possible additional move, there might be some edges that are not directed.  An \textit{unmarkable edge} is an edge that cannot be directed, as doing so would create either a sink or a source. Figure \ref{fig:unmarkable-edge-definition} provides an example of an unmarkable edge.

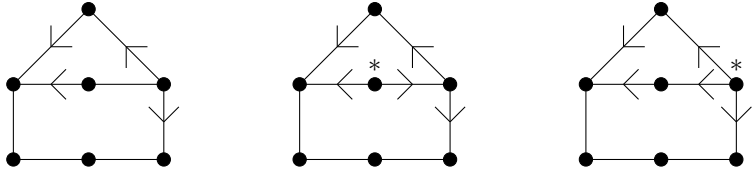
\begin{figure}
    \centering
    \begin{tikzpicture}
    \draw (-1,0)--(0,0);
    \draw [black] (0,0)--(1,0);
    \draw (-1,0)--(0,1);
    \draw (1,0)--(0,1);
    \draw (-1,0)--(-1,-1);
    \draw (1,0)--(1,-1);
    \draw (-1,-1)--(1,-1);
    \draw (0.5,0.5)--(0.7828,0.5);
    \draw (0.5,0.5)--(0.5,0.2172);
    \draw (-0.5,0.5)--(-0.5,0.7828);
    \draw (-0.5,0.5)--(-0.2172,0.5);
    \draw (-0.5,0)--(-0.3,0.2);
    \draw (-0.5,0)--(-0.3,-0.2);
    \draw (1,-0.5)--(0.8,-0.3);
    \draw (1,-0.5)--(1.2,-0.3);
    \draw [fill=black] (0,0) circle (2.5pt);
    \draw [fill=black] (1,0) circle (2.5pt);
    \draw [fill=black] (-1,0) circle (2.5pt);
    \draw [fill=black] (0,1) circle (2.5pt);
    \draw [fill=black] (-1,-1) circle (2.5pt);
    \draw [fill=black] (1,-1) circle (2.5pt);
    \draw [fill=black] (0,-1) circle (2.5pt);
    \end{tikzpicture}
    \hspace{0.5in}
    \begin{tikzpicture}
    \draw (-1,0)--(0,0);
    \draw [black] (0,0)--(1,0);
    \draw (-1,0)--(0,1);
    \draw (1,0)--(0,1);
    \draw (-1,0)--(-1,-1);
    \draw (1,0)--(1,-1);
    \draw (-1,-1)--(1,-1);
    \draw (0.5,0.5)--(0.7828,0.5);
    \draw (0.5,0.5)--(0.5,0.2172);
    \draw (-0.5,0.5)--(-0.5,0.7828);
    \draw (-0.5,0.5)--(-0.2172,0.5);
    \draw (-0.5,0)--(-0.3,0.2);
    \draw (-0.5,0)--(-0.3,-0.2);
    \draw (1,-0.5)--(0.8,-0.3);
    \draw (1,-0.5)--(1.2,-0.3);
    \draw [fill=black] (0,0) circle (2.5pt);
    \draw [fill=black] (1,0) circle (2.5pt);
    \draw [fill=black] (-1,0) circle (2.5pt);
    \draw [fill=black] (0,1) circle (2.5pt);
    \draw [fill=black] (-1,-1) circle (2.5pt);
    \draw [fill=black] (1,-1) circle (2.5pt);
    \draw [fill=black] (0,-1) circle (2.5pt);
    \draw (0,0.2) node[] {*};
    \draw (0.5,0)--(0.3,0.2);
    \draw (0.5,0)--(0.3,-0.2);
    \end{tikzpicture}
    \hspace{0.5in}
    \begin{tikzpicture}
    \draw (-1,0)--(0,0);
    \draw [black] (0,0)--(1,0);
    \draw (-1,0)--(0,1);
    \draw (1,0)--(0,1);
    \draw (-1,0)--(-1,-1);
    \draw (1,0)--(1,-1);
    \draw (-1,-1)--(1,-1);
    \draw (0.5,0.5)--(0.7828,0.5);
    \draw (0.5,0.5)--(0.5,0.2172);
    \draw (-0.5,0.5)--(-0.5,0.7828);
    \draw (-0.5,0.5)--(-0.2172,0.5);
    \draw (-0.5,0)--(-0.3,0.2);
    \draw (-0.5,0)--(-0.3,-0.2);
    \draw (1,-0.5)--(0.8,-0.3);
    \draw (1,-0.5)--(1.2,-0.3);
    \draw [fill=black] (0,0) circle (2.5pt);
    \draw [fill=black] (1,0) circle (2.5pt);
    \draw [fill=black] (-1,0) circle (2.5pt);
    \draw [fill=black] (0,1) circle (2.5pt);
    \draw [fill=black] (-1,-1) circle (2.5pt);
    \draw [fill=black] (1,-1) circle (2.5pt);
    \draw [fill=black] (0,-1) circle (2.5pt);
    \draw (1,0.2) node[] {*};
    \draw (0.5,0)--(0.7,0.2);
    \draw (0.5,0)--(0.7,-0.2);
    \end{tikzpicture}
    \caption{Either direction would create a source at the starred vertices, so the undirected edge on the middle path is unmarkable}
    \label{fig:unmarkable-edge-definition}
\end{figure}

\bigskip
\noindent Let's play the Game of Cycles with the gameboard in Figure \ref{fig:sample-gameboard} below in order to better understand the game. We will call the two edges paired if they are represented by the same type of lines. We show that Player 2 has a winning strategy by responding to Player 1's first move by playing the paired edge in the same direction.

\bigskip 
\noindent First notice that the representation of the graph matters in this game, so we will need to consider two cases for Player 1's first move.

\bigskip
\noindent \textbf{Case 1:} If Player 1 makes a move on the outer perimeter, say $b \rightarrow a$, then Player 2 should make the move on the paired edge. In our example, it is $a \rightarrow d$. From here, Player 1 must avoid the moves $c \rightarrow b$ and $d \rightarrow c$ since Player 2 could make the other of these moves to complete a cycle cell and win the game. So, there are three subcases for Player 1's next move.

\bigskip
\noindent \textbf{Subcase a)} If Player 1 makes the move $b \rightarrow c$, then Player 2 should make the move $c \rightarrow d$. From here, Player 1 cannot make either of the moves $b \rightarrow e$ or $e \rightarrow d$, as this would make a source at vertex $b$ or a sink at $d$. So, Player 1 is forced to make the move $d \rightarrow e$ or $e \rightarrow b$, which means Player 2 can make the other move and complete a cycle cell to win the game.

\bigskip
\noindent \textbf{Subcase b) }If Player 1 makes the move $b \rightarrow e$, then Player 2 should make the move $e \rightarrow d$. From here, Player 1 cannot make either of the moves $b \rightarrow c$ or $c \rightarrow d$, as this would make a source at vertex $b$ or a sink at $d$. So, Player 1 is forced to make the move $d \rightarrow c$ or $c \rightarrow b$, which means Player 2 can make the other move and complete a cycle cell to win the game.

\bigskip
\noindent \textbf{Subcase c)} If Player 1 makes the move $e \rightarrow b$, then Player 2 should make the move $d \rightarrow e$. From here, all of the moves $b \rightarrow c$, $c \rightarrow d$, $d \rightarrow c$, and $c \rightarrow b$ are losing moves for Player 1, as Player 2 can make the last remaining legal move to complete a cycle cell and win the game.

\bigskip
\noindent \textbf{Case 2:} If Player 1's first move is along the inside, say  $b \rightarrow c$, then Player 2 should make the move $c \rightarrow d$. From here, Player 1 must avoid the moves $a \rightarrow b$, $d \rightarrow a$, $e \rightarrow b$, and $d \rightarrow e$ since Player 2 could make a move to complete a cycle cell and win the game. Therefore, Player 1 must must direct one of the outside edges in the opposite direction, say $b \rightarrow e$. So, Player 2 should make the corresponding move $e \rightarrow d$. From here, the only way for Player 1 to avoid creating a sink or a source is to make the move $d \rightarrow a$ or $a \rightarrow b$, which means Player 2 can make the other move and complete a cycle cell to win the game.

\bigskip
\noindent Therefore, for this gameboard Player 2 has a winning strategy.

\begin{figure}
    \centering
    \begin{tikzpicture}
    \draw [dashed] (-1,0)--(1,0);
    \draw (-1,0)--(0,1);
    \draw (1,0)--(0,1);
    \draw [dotted] (-1,0)--(0,-1);
    \draw [dotted] (1,0)--(0,-1);
    \draw [fill=black] (0,0) circle (2.5pt);
    \draw [fill=black] (1,0) circle (2.5pt);
    \draw [fill=black] (-1,0) circle (2.5pt);
    \draw [fill=black] (0,1) circle (2.5pt);
    \draw [fill=black] (0,-1) circle (2.5pt);
    \draw (0,1.3) node[] {$a$};
    \draw (-1,0.3) node[] {$b$};
    \draw (0,0.3) node[] {$c$};
    \draw (1,0.3) node[] {$d$};
    \draw (0,-0.7) node[] {$e$};
    \end{tikzpicture}
    \caption{Example gameboard, with solid edges on top, dashed edges in middle, dotted edges on bottom}
    \label{fig:sample-gameboard}
\end{figure}
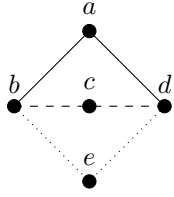

\bigskip
\noindent Previous results and winning strategies described in the article \textit{The Game of Cycles}, \cite{alvarado2021game}, include the following:

\begin{itemize}
    \item Player 2 has a winning strategy for graph $K_4$
    \item The number of unmarkable edges of a polygonal gameboard must be even
    \item Player 1 always wins polygons of odd length, and Player 2 always wins polygons of even length
    \item Winning strategies for a gameboard composed of two polygons that share exactly one edge (i.e., $\Theta_{j,1,k}$ graphs with $j \geq 1, k \geq 2$)
    \item Winning strategies for select symmetric gameboards
\end{itemize}

\begin{figure}
    \centering
    \begin{tikzpicture}
    \draw (0,0)--(0,1);
    \draw (0,0)--(-0.7071,-0.7071);
    \draw (0,0)--(0.7071,-0.7071);
    \draw (0,1)--(-0.7071,-0.7071);
    \draw (0,1)--(0.7071,-0.7071);
    \draw (-0.7071,-0.7071)--(0.7071,-0.7071);
    \draw [fill=black] (0,0) circle (2.5pt);
    \draw [fill=black] (0,1) circle (2.5pt);
    \draw [fill=black] (-0.7071,-0.7071) circle (2.5pt);
    \draw [fill=black] (0.7071,-0.7071) circle (2.5pt);
    \end{tikzpicture}
    \hspace{0.5in}
    \begin{tikzpicture}
    \draw (0,0)--(0,1);
    \draw (0,0)--(1,0);
    \draw (1,0)--(1,1);
    \draw [dashed] (0,1)--(1,1);
    \draw [fill=black] (0,0) circle (2.5pt);
    \draw [fill=black] (0,1) circle (2.5pt);
    \draw [fill=black] (1,0) circle (2.5pt);
    \draw [fill=black] (1,1) circle (2.5pt);
    \end{tikzpicture}
    \hspace{0.5in}
    \begin{tikzpicture}
    \draw (0,0)--(1,0);
    \draw (0,0.75)--(0,-0.75);
    \draw (1,0.75)--(1,-0.75);
    \draw [dashed] (0,0.75)--(1,0.75);
    \draw [dashed] (0,-0.75)--(1,-0.75);
    \draw [fill=black] (0,0) circle (2.5pt);
    \draw [fill=black] (1,0) circle (2.5pt);
    \draw [fill=black] (0,0.75) circle (2.5pt);
    \draw [fill=black] (0,-0.75) circle (2.5pt);
    \draw [fill=black] (1,0.75) circle (2.5pt);
    \draw [fill=black] (1,-0.75) circle (2.5pt);
    \end{tikzpicture}
    \hspace{0.5in}
    \begin{tikzpicture}
    \draw (-1,0)--(1,0);
    \draw (-1,0)--(-1,-0.75);
    \draw (1,0)--(1,-0.75);
    \draw (-1,0)--(-0.5,0.75);
    \draw (1,0)--(0.5,0.75);
    \draw (-0.5,0.75)--(0.5,0.75);
    \draw (-1,-0.75)--(-0.5,-1.25);
    \draw (1,-0.75)--(0.5,-1.25);
    \draw (-0.5,-1.25)--(0.5,-1.25);
    \draw [fill=black] (0,0) circle (2.5pt);
    \draw [fill=black] (-1,0) circle (2.5pt);
    \draw [fill=black] (1,0) circle (2.5pt);
    \draw [fill=black] (-0.5,0.75) circle (2.5pt);
    \draw [fill=black] (0.5,0.75) circle (2.5pt);
    \draw [fill=black] (-1,-0.75) circle (2.5pt);
    \draw [fill=black] (1,-0.75) circle (2.5pt);
    \draw [fill=black] (-0.5,-1.25) circle (2.5pt);
    \draw [fill=black] (0.5,-1.25) circle (2.5pt);
    \end{tikzpicture}
    \caption{Various gameboards, the first three solved by Alvarado et al.}
    \label{fig:various-gameboards}
\end{figure}
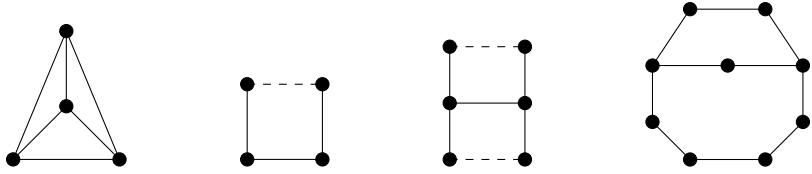

\bigskip
\noindent In this paper, we extend some of the above results, as well as attempt to solve some of the questions from \cite{alvarado2021game}. One such question that has motivated our research is solving the gameboard that is a pentagon and a heptagon that share two consecutive edges, as pictured in Figure \ref{fig:various-gameboards} ($\Theta_{3,2,5}$ graph). We have then expanded this research to consider generalized Theta graphs. 

\section{Lines and Basic Line Interactions}

Before analyzing various classes and structures of gameboards, we are going to consider the structure of the most basic possible gameboard, a single line, or \textit{path} between two vertices. Although a path has no cycles, many gameboards can be analyzed by considering paths of variable length. Therefore, it is valuable to know the number of unmarkable edges of a path with as little information as possible. Note that when we play this game on a single path, we ignore the sink-source rule for the end vertices.

\bigskip
\noindent The number of unmarkable edges of a path depends on the direction of the first and last edges of the path.

\begin{lemma}\label{thm:unmarkable-edges-line}
A path will have an even number of unmarkable edges when the two end edges are directed the same direction and an odd number of unmarkable edges when the two end edges are directed in opposite directions.
\end{lemma}

\begin{proof}
Imagine we add an edge to the path between the two end vertices. What we have is a polygon with one more edge than the original path. Take note of which edge was not part of the original path. We know that polygons must have an even number of unmarkable edges \cite{alvarado2021game}. 

\bigskip
\noindent If the two end edges are directed in the same direction, then the new edge will not be unmarkable. So, since there must be an even number of unmarkable edges overall, there are an even number of unmarkable edges on the original path.

If the two end edges are directed in opposite directions, then the new edge will be unmarkable. So, since there must be an even number of unmarkable edges overall, there are an odd number of unmarkable edges on the original path.
\end{proof}

\bigskip
\noindent Let's consider playing the Game of Cycles on a path. Since playing the game on a path does not have any cycles, the winner is the player who plays the last legal move. We can use Lemma \ref{thm:unmarkable-edges-line} as part of our strategy, as knowing the parity of unmarkable edges will tell us which player can make the last move. 

\bigskip
\noindent Throughout this paper we will often concentrate on edges incident to a particular vertex. For ease of discussion we introduce two new definitions.
We call a \textit{primary edge} an edge that is incident to a vertex of interest (which will define when encountering a new graph), and a \textit{secondary edge} an edge that is adjacent to a primary edge, but is not itself a primary edge. 

\bigskip
\noindent For a path, the vertices of interest are the end vertices. Notice that directing a secondary edge has the same effect as directing its adjacent primary edge, as in order to follow the sink-source rule the adjacent primary edge may only be directed one way. So, for paths with five edges or more, the player who marks a primary or secondary edge first will lose the game, as the other player can follow up by playing the primary edge on the opposite side of the path in the direction that would guarantee the appropriate number of unmarkable edges. Consider an example of this in Figure \ref{fig:primary-secondary-edges}. Player 1 has made the move $b \rightarrow c$. However, the edge $bc$ is a secondary edge, so Player 2 can follow up by playing the move $f \rightarrow e$. This guarantees that the primary edges will point in opposite directions, meaning there will be odd number of unmarkable edges total. Since the path has five edges to start with, there will be four edges that are marked throughout the game, meaning Player 2 will make the final legal move and win. We extended this result to find the winner for a game of any length of path.

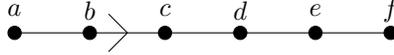
\begin{figure}
    \centering
    \begin{tikzpicture}
    \draw (-3,0)--(2,0);
    \draw (-1.5,0)--(-1.75,0.25);
    \draw (-1.5,0)--(-1.75,-0.25);
    \draw [fill=black] (-3,0) circle (2.5pt);
    \draw [fill=black] (-2,0) circle (2.5pt);
    \draw [fill=black] (-1,0) circle (2.5pt);
    \draw [fill=black] (0,0) circle (2.5pt);
    \draw [fill=black] (1,0) circle (2.5pt);
    \draw [fill=black] (2,0) circle (2.5pt);
    \draw (-3,0.3) node[] {$a$};
    \draw (-2,0.3) node[] {$b$};
    \draw (-1,0.3) node[] {$c$};
    \draw (0,0.3) node[] {$d$};
    \draw (1,0.3) node[] {$e$};
    \draw (2,0.3) node[] {$f$};
    \end{tikzpicture}
    \caption{Player 1 has made a losing move in this game}
    \label{fig:primary-secondary-edges}
\end{figure}

\begin{theorem}\label{thm:winning-strategy-line}
Assume we are playing the Game of Cycles on a path, and say Player W is the player with the winning strategy. If the path has an odd number of edges Player 1 is Player W, and if the path has an even number of edges Player 2 is Player W.
\end{theorem}

\begin{proof}
For lines with one edge or three edges, Player 1 obviously has a winning strategy by playing the only edge or playing the center edge, respectively.

\bigskip
\noindent For lines of two edges or four edges, Player 2 obviously has a winning strategy by playing the only available edge or by orienting the opposite primary edge in the same direction as the last made move, respectively.

\bigskip
\noindent Assume that Player W has a winning strategy for a path of $\ell$ edges. Consider a path of length $\ell + 4$. We may imagine that such a path is actually the path of length $\ell$ with two additional edges on either side. Those four additional edges are the two primary edges and the two secondary edges.

\bigskip
\noindent In order to win, Player $W$ wants to play the last legal move on the inner $\ell$ edges, forcing the other player to make the first move on one of the primary or secondary edges. We assumed that Player W has a strategy to win the line of length $\ell$ which would mean they have a strategy to play the last legal move on $\ell$ edges. Thus, Player W can force the other player to play on a primary or secondary edge first, and then orient the opposite primary edge accordingly.

\bigskip
\noindent Of course, if the other player plays a primary or secondary edge earlier than this, Player $W$ must play the opposite primary edge accordingly.

\bigskip
\noindent Therefore, by induction, Player 1 is Player W for a path with and odd number edges and Player 2 is Player W for a path with an even number of edges.
\end{proof}

\bigskip
\noindent We will use this theorem about the game on a path to analyze other classes of gameboards. Notice that on a path of three edges, Player 1 is the first player to direct a primary or secondary edge, but still wins the game, which goes against our earlier intuition. This is because the center edge for paths of length three is \textit{doubly secondary}, as it is secondary to two different primary edges. This fact makes analyzing a game when there are multiple short paths difficult.

\bigskip
Finally, we are also interested how these same primary and secondary edges can interact between paths. Say an \textit{explicit signal} is the direction given by a primary edge and an \textit{implicit signal} is the direction that an unmarkable primary edge should be according to the other edges incident with the vertex of interest. We will use -1 to represent a signal directed away from a degree 3 vertex, and use 1 to represent a signal directed toward a degree 3 vertex. 

\bigskip
\noindent Consider Case B of Figure \ref{fig:guaranteed-even-unmarkable-edges}. The bottom left edge of this case is not directed. Suppose that is is unmarkable. However, if we determined the direction of this edge based on the directed adjacent edges pictured, we would have to direct it way from the degree 3 vertex in order to avoid creating a sink. This is the implicit signal of this edge. So, for the left degree 3 vertex of this case there are two explicit 1 signals and one implicit -1 signal.

\bigskip
\noindent We define an \textit{interaction} of a path, $\{a,b\}$, the multiset of the two signals of the two primary edges of the path. There are three types of interactions: $\{1,1\}$, $\{-1,-1\}$, and $\{1,-1\}$. When both of the signals given are explicit signals, the number of unmarkable edges of the path is given by Lemma \ref{thm:unmarkable-edges-line}. That is, for explicit signals the $\{1,1\}$ and $\{-1,-1\}$ interactions will result in an odd number of unmarkable edges, and the $\{1,-1\}$ interaction results in an even number of unmarkable edges. We will show that the same result holds when one or both of the signals are implicit.

\bigskip
\noindent We denote the parity of the number of unmarkable edges by $p$.

\begin{lemma}\label{thm:implicit-signals-same}
Implicit signals and explicit signals in  an interaction of a path with at least $4$ edges result in the same parity of the number of unmarkable edges of the path.
\end{lemma}

\begin{proof}
We handle this proof with two cases regarding the number of implicit signals in the interaction. We need the length requirement in order to assure that no edge sends signals to both end vertices.

\bigskip
\noindent \textbf{Case 1: One implicit signal}

\bigskip
\noindent  Since there is one implicit signal, then there is already one unmarkable primary edge. Consider the smaller path which excludes this primary edge. This smaller path has an interaction in which both of the signals are explicit, and the parity of the number of unmarkable edges is $p + 1 \mod 2$. This is because we know the secondary edge adjacent to the unmarkable primary edge has the opposite signal than that given by the primary edge, so the parity of the number of unmarkable edges will also be the opposite.

\bigskip
\noindent Hence, the parity of unmarkable edges for the whole path is $((p + 1 \mod 2) + 1) \mod 2 = p \mod 2$. Therefore, the parity of unmarkable edges is unchanged from that of an interaction with two explicit signals.

\bigskip
\noindent \textbf{Case 2: Two implicit signals}

\bigskip
\noindent Since there are two implicit signals, there are two unmarkable primary edges. Consider the smaller path which excludes these primary edges. This smaller path has an interaction in which both of the signals are explicit, and the parity of the number of unmarkable edges is $p \mod 2$. This is because we know the secondary edges adjacent to the unmarkable primary edges have the opposite signal than that given by the primary edges, but since the signal is switched on both sides, there will be the same parity of unmarkable edges on this smaller path.

\bigskip
\noindent So, the parity of unmarkable edges for the whole path is $((p \mod 2) + 2) \mod 2 = p \mod 2$. Therefore, the parity of unmarkable edges is unchanged from that of an interaction with two explicit signals.

\end{proof}

\noindent It is not difficult to show that this lemma also holds for paths of length one, two, and three.

\section{Generalized Theta graphs}

\subsection{Structure of the Graph}

Plane Theta graphs have the structure of $p$ paths between two vertices $u$ and $v$. For the sake of ease of the discussion we will require that $u$ and $v$ are the leftmost and rightmost vertices respectively, as seen in Figure \ref{fig:p-paths-graph}. We will also require that the graph is symmetric across a vertical line of symmetry.

\begin{figure}
    \centering
    \begin{tikzpicture}
    \draw (-3,1)--(3,1);
    \draw (-3.5,1.5)--(-3,1);
    \draw (3,1)--(3.5,1.5);
    \draw (-3.5,1.5)--(3.5,1.5);
    \draw (-3.5,1.5)--(-3,2);
    \draw (-1,2)--(1,2);
    \draw (-1,2)--(-3.5,1.5);
    \draw (1,2)--(3.5,1.5);
    \draw (-3,2)--(-3,2.5);
    \draw (-3,2.5)--(3,2.5);
    \draw (3,2.5)--(3,2);
    \draw (3.5,1.5)--(3,2);
    \draw [dashed] (0,3)--(0,0.5);
    \draw [fill=black] (1,2) circle (2.5pt);
    \draw [fill=black] (-3,2) circle (2.5pt);
    \draw [fill=black] (-3,1) circle (2.5pt);
    \draw [fill=black] (-1,1) circle (2.5pt);
    \draw [fill=black] (1,1) circle (2.5pt);
    \draw [fill=black] (3,1) circle (2.5pt);
    \draw [fill=black] (-1,2) circle (2.5pt);
    \draw [fill=black] (3,2) circle (2.5pt);
    \draw [fill=black] (-3,2.5) circle (2.5pt);
    \draw [fill=black] (-1,2.5) circle (2.5pt);
    \draw [fill=black] (3,2.5) circle (2.5pt);
    \draw [fill=black] (1,2.5) circle (2.5pt);
    \draw [fill=black] (0.,1.5) circle (2.5pt);
    \draw [fill=white] (-3.5,1.5) circle (2.5pt);
    \draw [fill=white] (3.5,1.5) circle (2.5pt);
    \draw (-4,1.5) node[] {$u$};
    \draw (4,1.5) node[] {$v$};
    \end{tikzpicture}
    \caption{Example of a path of $p$ paths}
    \label{fig:p-paths-graph}
\end{figure}
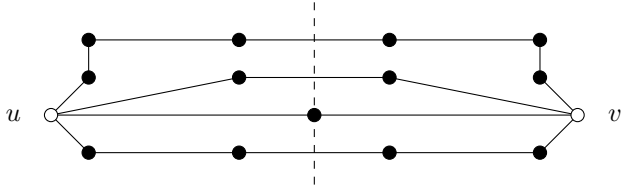

\bigskip
\noindent This second requirement allows us to define two different types of edges. An \textit{unmirrorable edge} is an edge that is self-reflective under the vertical line of symmetry. These are the edges in Figure \ref{fig:p-paths-graph} in which the vertical line of symmetry goes through the edge. A \textit{mirrorable edge} is an edge that is not self-reflective under the vertical line of symmetry. These are all of the other edges of Figure \ref{fig:p-paths-graph}.

\subsection{Strategy for generalized Theta graphs}

\bigskip
\noindent In the descriptions of winning strategies we will describe a common response to a given move as follows.

\bigskip
\noindent The \textit{mirror-reverse move} is the move made by flipping the previous move over the vertical line of symmetry, and then reversing the direction of the newly played edge. An example of the mirror-reverse move is shown in Figure \ref{fig:mirror-reverse-move}. We note that the mirror-reverse move is only a possible response if the previous move was made on a mirrorable edge.

\begin{figure}
    \centering
    \begin{tikzpicture}
    \draw (-2,0)--(3,0);
    \draw (-1.5,0)--(-1.75,0.25);
    \draw (-1.5,0)--(-1.75,-0.25);
    \draw [fill=black] (-2,0) circle (2.5pt);
    \draw [fill=black] (-1,0) circle (2.5pt);
    \draw [fill=black] (0,0) circle (2.5pt);
    \draw [fill=black] (1,0) circle (2.5pt);
    \draw [fill=black] (2,0) circle (2.5pt);
    \draw [fill=black] (3,0) circle (2.5pt);
    \draw [white] (0,-0.75) circle (2.5pt);
    \end{tikzpicture}
    \hspace{4in}
    \begin{tikzpicture}
    \draw (-2,0)--(3,0);
    \draw (-1.5,0)--(-1.75,0.25);
    \draw (-1.5,0)--(-1.75,-0.25);
    \draw (2.5,0)--(2.75,0.25);
    \draw (2.5,0)--(2.75,-0.25);
    \draw [fill=black] (-2,0) circle (2.5pt);
    \draw [fill=black] (-1,0) circle (2.5pt);
    \draw [fill=black] (0,0) circle (2.5pt);
    \draw [fill=black] (1,0) circle (2.5pt);
    \draw [fill=black] (2,0) circle (2.5pt);
    \draw [fill=black] (3,0) circle (2.5pt);
    \draw [white] (0,-0.75) circle (2.5pt);
    \end{tikzpicture}
    \hspace{4in}
    \begin{tikzpicture}
    \draw (-2,0)--(3,0);
    \draw (-1.5,0)--(-1.75,0.25);
    \draw (-1.5,0)--(-1.75,-0.25);
    \draw (2.5,0)--(2.25,0.25);
    \draw (2.5,0)--(2.25,-0.25);
    \draw [fill=black] (-2,0) circle (2.5pt);
    \draw [fill=black] (-1,0) circle (2.5pt);
    \draw [fill=black] (0,0) circle (2.5pt);
    \draw [fill=black] (1,0) circle (2.5pt);
    \draw [fill=black] (2,0) circle (2.5pt);
    \draw [fill=black] (3,0) circle (2.5pt);
    \end{tikzpicture}
    \caption{A move, its mirror move, and the mirror-reverse move}
    \label{fig:mirror-reverse-move}
\end{figure}
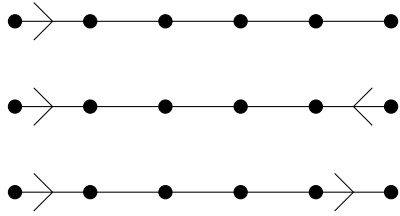

\bigskip
\noindent As explored in \cite{alvarado2021game}, continually playing the mirror-reverse move guarantees that later mirror-reverse moves will always be available. That is, if the mirror-reverse move would be illegal for some reason, it must be that the previous move was also in fact illegal. This is because mirror-reverse moves guarantee a structure that replicates direction changes on a path.

\bigskip
\noindent We also define two more types of edges that are used in these winning strategies. We say that two edges are \textit{paired} in the game if one of them is the edge played by the Winner in response to a move on the other edge by the other player. At the beginning of the game, two edges are paired when they are each others' reflection over the vertical line of symmetry, so all of the mirrorable edges are paired edges.

\bigskip
\noindent An edge is a \textit{loner edge} if the edge is not paired. That is, in the winning strategy the response to a move on a loner edge is not an edge on the same path. At the beginning of the game, the loner edges are the unmirrorable edges.

\bigskip
\noindent It is possible that a loner edge may become a paired edge, or that a paired edge may become a loner edge. This will become important when at some point in the game there are exactly three undirected markable edges on a path, and one of these edges is a loner edge.  Then, if the Loser directs one of the paired edges, we redefine the pairing of these edges. We say that the two remaining unmarked edges are now paired edges, and the edge that was just marked was a loner edge. In Figure \ref{fig:paired-loner-switch-roles} we have the initial conditions of this case, the move that is made in this case, and the new set of paired edges. It is easy to see that this change of pairing can only happen on paths with an odd number of edges. We will call this strategy \textit{re-pairing}.

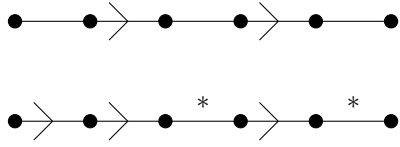
\begin{figure}
    \centering
    \begin{tikzpicture}
    \draw (-2,0)--(3,0);
    \draw (-0.5,0)--(-0.75,0.25);
    \draw (-0.5,0)--(-0.75,-0.25);
    \draw (1.5,0)--(1.25,0.25);
    \draw (1.5,0)--(1.25,-0.25);
    \draw [fill=black] (-2,0) circle (2.5pt);
    \draw [fill=black] (-1,0) circle (2.5pt);
    \draw [fill=black] (0,0) circle (2.5pt);
    \draw [fill=black] (1,0) circle (2.5pt);
    \draw [fill=black] (2,0) circle (2.5pt);
    \draw [fill=black] (3,0) circle (2.5pt);
    \draw [white] (0,-0.75) circle (2.5pt);
    \end{tikzpicture}
    \hspace{4in}
    \begin{tikzpicture}
    \draw (-2,0)--(3,0);
    \draw (-1.5,0)--(-1.75,0.25);
    \draw (-1.5,0)--(-1.75,-0.25);
    \draw (-0.5,0)--(-0.75,0.25);
    \draw (-0.5,0)--(-0.75,-0.25);
    \draw (1.5,0)--(1.25,0.25);
    \draw (1.5,0)--(1.25,-0.25);
    \draw [fill=black] (-2,0) circle (2.5pt);
    \draw [fill=black] (-1,0) circle (2.5pt);
    \draw [fill=black] (0,0) circle (2.5pt);
    \draw [fill=black] (1,0) circle (2.5pt);
    \draw [fill=black] (2,0) circle (2.5pt);
    \draw [fill=black] (3,0) circle (2.5pt);
    \draw (0.5,0.2) node[] {*};
    \draw (2.5,0.2) node[] {*};
    \end{tikzpicture}
    \caption{Top image provides initial conditions, bottom image provides new move, and stars newly paired edges}
    \label{fig:paired-loner-switch-roles}
\end{figure}

Now that we have made these definitions, we may proceed with our result.

\begin{theorem}\label{thm:p-paths}
Consider a generalized Theta graph consisting of $p$ disjoint paths between two vertices $u$ and $v$ with the property that every path with an odd number of edges between $u$ and $v$ has at least 5 edges. For such a graph, Player 1 has a winning strategy when there is an odd number of total edges and Player 2 has a winning strategy when there is an even number of total edges.
\end{theorem}

\begin{proof}
The winning strategy will be described based on the the total number of edges and the opponent's moves. Note that while a game may change from Case I.1 to Case I.2, the opposite may not happen.

\bigskip
\noindent \textbf{Case I} \textbf{Odd number of total edges}

\bigskip
\noindent Player 1 's first move is directing a loner edge.
    
\bigskip
\noindent \textbf{I.1) Start of the game}: \textbf{There is at least one path whose primary edges have or could be directed to have an explicit signal from left to right and at least one path whose primary edges have or could be directed to have an explicit signal from right to left}

\bigskip
\noindent Note that every game with at least 2 paths starts in this configuration. While the conditions of this case are satisfied, the information given by the primary edges suggest that a cycle (note: not necessarily a cycle cell) is possible.
        
\bigskip
\noindent While this condition is satisfied, we will show that Player 1 may make the last move on every path. Therefore, if the winning player wins with a cycle cell, Player 1 will complete the cycle cell by completing the path, and if the winning player wins by making the last move, Player 1 will make the last move on every path, and therefore on the graph as well.

\bigskip
\noindent If \textbf{Player 2 plays a loner edge} Player 1's response is to play a loner edge. Because there are an odd number of edges overall, there must be an odd number of paths with an odd number of edges, and therefore an odd number of loner edges at the beginning of the game. Since Player 1 started by playing a loner edge, if they play a loner edge whenever Player 2 plays a loner edge, Player 1 will direct the last loner edge.This also includes the games with re-pairings. 
            
 \bigskip
\noindent If \textbf{Player 2 plays a paired edge} we have two possibilities.
            
\bigskip
\noindent If the played edge is on a path with an even number of edges or the played edge is on a path with an odd number of edges and the unmirrorable edge is not a paired edge then Player 1 can play the  mirror-reverse move. 

\bigskip
\noindent If the played edge is on a path with an odd number of edges and the unmirrorable edge is a paired edge then there is only one other move available on this path, namely, playing the edge paired to the edge just played. 

\bigskip
\noindent At some point in the game we might reach the opposite situation, and then the game continues with the following configuration:
        
\bigskip
\noindent  \textbf{I.2) End Game}  

\bigskip
\noindent If the prescribed conditions are no longer satisfied it is either impossible to direct the primary edges of any path from left to right, or it is impossible to direct the primary edges of any path from right to left. Without loss of generality, consider the the case when no pair of the primary edges may be directed from left to right. 
        
\bigskip
\noindent  The key fact to note regarding this case is that since no pair of primary edges may be or have been directed from left to right, it is impossible to create a cycle cell. So, all that is left to show is that there must be an even number of unmarkable edges overall.
        
\bigskip
\noindent Since this configuration is reached while Player 1 used the generalized mirror reverse strategy described above, this means that none of the primary edges are or could be directed from left to right. Therefore, the pair of primary edges on any path are either both directed from right to left, or at least one of the pair is unmarkable. 

However, not all of the primary edges may be directed from right to left, as this would create a sink at vertex $u$ and a source at vertex $v$. Therefore, we know that there is at least one unmarkable edge incident to vertex $u$ and at least one unmarkable edge incident to vertex $v$.  

Notice that for any given vertex, there may be only one unmarkable edge incident to that vertex, as if there are two undirected edges incident to the vertex, one of them can be directed without creating a sink or a source, simply meeting the needs of the other vertex incident that edge. Therefore, there is exactly one unmarkable edge incident to vertex $u$ and exactly one unmarkable edge incident to vertex $v$.
        
\bigskip
\noindent We know that all but two of the primary edges are directed from right to left. Using Lemma \ref{thm:unmarkable-edges-line}, any path that has both primary edges directed from right to left has an even number of unmarkable edges. 

\bigskip
\noindent The two unmarkable primary edges must have an implicit signal from the vertices $u$ and $v$ directing them from left to right. Whether these two primary edges are on the same path or on different paths, we may apply Lemma \ref{thm:unmarkable-edges-line} and Lemma \ref{thm:implicit-signals-same} to conclude there is an even number of unmarkable edges on this one, or on these two paths together. In either case, there is an even number of unmarkable edges overall.
        
\bigskip
\noindent  Therefore, the winning player is determined by the parity of the total number of edges.
        
\bigskip
\noindent  \textbf{Case II} \textbf{Even number of total edges} 
    
\bigskip
\noindent  In this case, Player 2 should follow the strategy described in Case I. The only different argument needed is that because there are an even number of edges overall, there must be an even number of paths with an odd number of edges, and therefore an even number of loner edges. So, if Player 2 plays a loner edge whenever Player 1 plays a loner edge, they will direct the last loner edge.

\end{proof}

\begin{corollary}\label{cor:p-paths-extended}
Consider a generalized Theta graph consisting of $p$ disjoint paths between two vertices $u$ and $v$ with the property that every path with an odd number of edges between $u$ and $v$ has at least 5 edges, except for graphs that have an odd number of total edges, there may be a single path of length 3 or a single path of length 1. For such a graph, Player 1 has a winning strategy when there is an odd number of total edges and Player 2 has a winning strategy when there is an even number of total edges.

\end{corollary}

\begin{proof}
If there is a short odd path then Player 1's first move will be on the central, or only edge of the short odd path path. From this point, they may follow the strategy as described in the proof above. 
\end{proof}

\subsection{Issue of Short Odd Paths}

There is an issue with having many short paths with odd number of vertices because when Player 1 does not have a chance to play these unmirrorable edges early on, and uses the above strategy, Player 2 may exploit their plans.

\bigskip
\noindent For example, consider the situation where a path with an even number of edges and a path of one edge could potentially be directed to result in a a cycle cell. If Player 1 does not have the chance to direct the single edge beforehand, Player 2 may play the edges on the path with even edges. Since Player 1's response is to play the mirror-reverse of these moves, Player 1 will complete the path, and then Player 2 may direct the path with one edge to complete the cycle cell.

\bigskip
\noindent For paths with three edges, there is again an issue if Player 1 cannot make sure they play the unmirrorable edge early on. If not, Player 2 may play one of the primary edges of the path with three edges. This has two potentially negative outcomes. If Player 1 leaves this path the way it is, Player 2 could later direct the other primary edge such that there is an odd number of unmarkable edges on the path. If Player 1 plays the mirror-reverse move, there could be a similar situation as described with the path with one edge.

\section{Specific Theta Graphs}

Motivated by the heptagon-pentagon graph mentioned previously, we are going to consider $\Theta_{j,2,k}$ graphs. These graphs could also be considered two polygons that share 2 consecutive sides. We structure the graph as it is structured in Figure \ref{fig:j-2-k-primary-edges}, and say the upper path has $j$ edges, the inner path has $2$ edges, and the lower path has $k$ edges. Note that only one of $j \text{ and } k$ can be equal to 1, as otherwise gameboard would include multiple edges, and would not be a simple graph. The example game we played earlier used $j = k = 2$. [See Figure \ref{fig:sample-gameboard}.] The unsolved graph of a heptagon and a pentagon that share two sides is the case when $j = 3$ and $k = 5$.

\bigskip
\noindent The vertices of interest in these $\Theta_{j,2,k}$ graphs are the degree-3 vertices. Ultimately, the way the primary edges are directed will determine the parity of the number of unmarkable edges in the graph, so we are interested in ways we may control these primary edges, seen in Figure \ref{fig:j-2-k-primary-edges}. We label these primary edges so we may reference them later on.

\begin{figure}
    \centering
    \begin{tikzpicture}
    \draw (-2.,4.)--(-2.,0.);
    \draw (-2.,2.)--(2.,2.);
    \draw (2.,4.)--(2.,0.);
    \draw [dashed] (-2, 4)--(2,4);
    \draw [dashed] (-2, 0)--(2,0);
    \draw [fill=black] (-2.,0.) circle (3.5pt);
    \draw [fill=black] (2.,0.) circle (3.5pt);
    \draw [fill=black] (-2.,4.) circle (3.5pt);
    \draw [fill=black] (2.,4.) circle (3.5pt);
    \draw [fill=black] (-2.,2.) circle (3.5pt);
    \draw [fill=black] (2.,2.) circle (3.5pt);
    \draw [fill=black] (0.,2.) circle (3.5pt);
    \draw (-2.3,3) node[] {1};
    \draw (-2.3,1) node[] {2};
    \draw (-1,2.3) node[] {3};
    \draw (1,2.3) node[] {4};
    \draw (2.3,1) node[] {5};
    \draw (2.3,3) node[] {6};
    \end{tikzpicture}
    \caption{Primary edges and structure of $\Theta_{j,2,k}$ gameboard}
    \label{fig:j-2-k-primary-edges}
\end{figure}
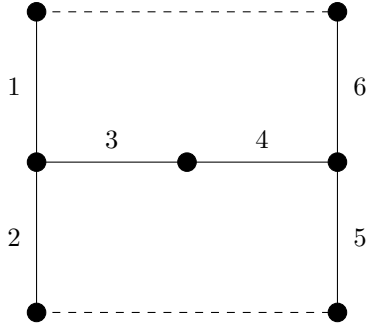

\subsection{Winning Configurations}

Our strategy consists of finding a group of configurations with the following properties: 

\begin{itemize}
    \item Reaching the given configuration guarantees an even number of unmarkable edges
    \item It is not possible to complete both cycle cells in the graph
    \item Both players can reach at least one of the given configurations regardless of the other player's moves
\end{itemize}

Once we found such a group of configurations we can conclude that Player 1 will win if the total number of edges is odd, and Player 2 will win if the total number of edges is even.

\bigskip
\noindent After investigating the possible directions that these primary edges could be labelled (as well as the possibility that they become unmarkable edges) we identified a few cases that always guarantee an even number of unmarkable edges. These cases are shown in Figure \ref{fig:guaranteed-even-unmarkable-edges}. There are more cases which also meet this requirement, but they are equivalent via rotational, horizontal, or vertical symmetry.

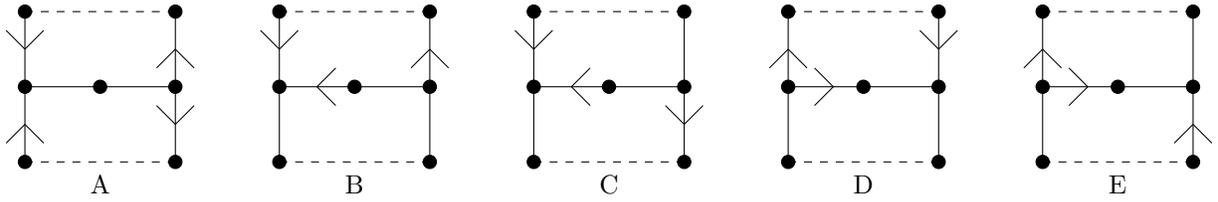
\begin{figure}
    \centering
    \begin{tikzpicture}
    \draw (-1,-1)--(-1,1);
    \draw (-1,0)--(1,0);
    \draw (1,1)--(1,-1);
    \draw [dashed] (-1,1)--(1,1);
    \draw [dashed] (-1,-1)--(1,-1);
    \draw [fill=black] (-1,1) circle (2.5pt);
    \draw [fill=black] (-1,0) circle (2.5pt);
    \draw [fill=black] (-1,-1) circle (2.5pt);
    \draw [fill=black] (0,0) circle (2.5pt);
    \draw [fill=black] (1,-1) circle (2.5pt);
    \draw [fill=black] (1,0) circle (2.5pt);
    \draw [fill=black] (1,1) circle (2.5pt);
    \draw (0,-1.3) node[] {A};
    
    \draw (-1,0.5)--(-1.25,0.75);
    \draw (-1,0.5)--(-0.75,0.75);
    
    \draw (-1,-0.5)--(-1.25,-0.75);
    \draw (-1,-0.5)--(-0.75,-0.75);
    
    \draw (1,0.5)--(0.75,0.25);
    \draw (1,0.5)--(1.25,0.25);
    
    \draw (1,-0.5)--(0.75,-0.25);
    \draw (1,-0.5)--(1.25,-0.25);
    \end{tikzpicture}
    \hspace{0.25in}
    \begin{tikzpicture}
    \draw (-1,-1)--(-1,1);
    \draw (-1,0)--(1,0);
    \draw (1,1)--(1,-1);
    \draw [dashed] (-1,1)--(1,1);
    \draw [dashed] (-1,-1)--(1,-1);
    \draw [fill=black] (-1,1) circle (2.5pt);
    \draw [fill=black] (-1,0) circle (2.5pt);
    \draw [fill=black] (-1,-1) circle (2.5pt);
    \draw [fill=black] (0,0) circle (2.5pt);
    \draw [fill=black] (1,-1) circle (2.5pt);
    \draw [fill=black] (1,0) circle (2.5pt);
    \draw [fill=black] (1,1) circle (2.5pt);
    \draw (0,-1.3) node[] {B};
    
    \draw (-1,0.5)--(-1.25,0.75);
    \draw (-1,0.5)--(-0.75,0.75);
    
    \draw (1,0.5)--(0.75,0.25);
    \draw (1,0.5)--(1.25,0.25);
    
    \draw (-0.5,0)--(-0.25,0.25);
    \draw (-0.5,0)--(-0.25,-0.25);
    \end{tikzpicture}
    \hspace{0.25in}
    \begin{tikzpicture}
    \draw (-1,-1)--(-1,1);
    \draw (-1,0)--(1,0);
    \draw (1,1)--(1,-1);
    \draw [dashed] (-1,1)--(1,1);
    \draw [dashed] (-1,-1)--(1,-1);
    \draw [fill=black] (-1,1) circle (2.5pt);
    \draw [fill=black] (-1,0) circle (2.5pt);
    \draw [fill=black] (-1,-1) circle (2.5pt);
    \draw [fill=black] (0,0) circle (2.5pt);
    \draw [fill=black] (1,-1) circle (2.5pt);
    \draw [fill=black] (1,0) circle (2.5pt);
    \draw [fill=black] (1,1) circle (2.5pt);
    \draw (0,-1.3) node[] {C};
    
    \draw (-1,0.5)--(-1.25,0.75);
    \draw (-1,0.5)--(-0.75,0.75);
    
    \draw (1,-0.5)--(0.75,-0.25);
    \draw (1,-0.5)--(1.25,-0.25);
    
    \draw (-0.5,0)--(-0.25,0.25);
    \draw (-0.5,0)--(-0.25,-0.25);
    \end{tikzpicture}
    \hspace{0.25in}
    \begin{tikzpicture}
    \draw (-1,-1)--(-1,1);
    \draw (-1,0)--(1,0);
    \draw (1,1)--(1,-1);
    \draw [dashed] (-1,1)--(1,1);
    \draw [dashed] (-1,-1)--(1,-1);
    \draw [fill=black] (-1,1) circle (2.5pt);
    \draw [fill=black] (-1,0) circle (2.5pt);
    \draw [fill=black] (-1,-1) circle (2.5pt);
    \draw [fill=black] (0,0) circle (2.5pt);
    \draw [fill=black] (1,-1) circle (2.5pt);
    \draw [fill=black] (1,0) circle (2.5pt);
    \draw [fill=black] (1,1) circle (2.5pt);
    \draw (0,-1.3) node[] {D};
    
    \draw (-1,0.5)--(-1.25,0.25);
    \draw (-1,0.5)--(-0.75,0.25);
    
    \draw (1,0.5)--(0.75,0.75);
    \draw (1,0.5)--(1.25,0.75);
    
    \draw (-0.4,0)--(-0.65,0.25);
    \draw (-0.4,0)--(-0.65,-0.25);
    \end{tikzpicture}
    \hspace{0.25in}
    \begin{tikzpicture}
    \draw (-1,-1)--(-1,1);
    \draw (-1,0)--(1,0);
    \draw (1,1)--(1,-1);
    \draw [dashed] (-1,1)--(1,1);
    \draw [dashed] (-1,-1)--(1,-1);
    \draw [fill=black] (-1,1) circle (2.5pt);
    \draw [fill=black] (-1,0) circle (2.5pt);
    \draw [fill=black] (-1,-1) circle (2.5pt);
    \draw [fill=black] (0,0) circle (2.5pt);
    \draw [fill=black] (1,-1) circle (2.5pt);
    \draw [fill=black] (1,0) circle (2.5pt);
    \draw [fill=black] (1,1) circle (2.5pt);
    \draw (0,-1.3) node[] {E};
    
    \draw (-1,0.5)--(-1.25,0.25);
    \draw (-1,0.5)--(-0.75,0.25);
    
    \draw (1,-0.5)--(0.75,-0.75);
    \draw (1,-0.5)--(1.25,-0.75);
    
    \draw (-0.4,0)--(-0.65,0.25);
    \draw (-0.4,0)--(-0.65,-0.25);
    \end{tikzpicture}
    \caption{Cases of $\Theta_{j,2,k}$ graphs that guarantee an even number of unmarkable edges}
    \label{fig:guaranteed-even-unmarkable-edges}
\end{figure}
 
\begin{lemma}\label{thm:proof-of-even-cases}
Each of the five configurations of Figure \ref{fig:guaranteed-even-unmarkable-edges} results in an even number of unmarkable edges.
\end{lemma}

\begin{proof}
Because of Lemma \ref{thm:implicit-signals-same}, as long as we know the signal that an edge will have in an interaction, we can predict the parity of unmarkable edges in  the given path. So, we need not worry in this proof if a signal is implicit or explicit. Throughout this proof we will use Lemma \ref{thm:unmarkable-edges-line} repeatedly without explicitly referencing it.

\bigskip
\noindent \textbf{Case A:} We know that each of the top and bottom paths must have an even number of unmarkable edges. Also, both of the edges on the middle path must be directed the same direction. So, overall there must be an even number of unmarkable edges.

\bigskip
\noindent \textbf{Case B:} Here the top path must have an even number of unmarkable edges and edge 2 must have a signal of -1. If edge 5 has a signal of 1, then the bottom path must have an even number of unmarkable edges. Then, edge 4 is not unmarkable. So, there are an even number of unmarkable edges overall. Instead, if edge 5 has a signal of -1, the number of unmarkable edges on the bottom path would be odd and edge 4 would become unmarkable, so, there would be an even number of unmarkable edges overall. 

\bigskip
\noindent \textbf{Case C:} In this case edge 2 will have a signal of -1, and the bottom path has an odd number of unmarkable edges. If edge 6 has a signal of -1, then the top path will have an even number of unmarkable edges and edge 4 would become unmarkable. So, there are an even number of unmarkable edges overall. Instead, if edge 6 has a signal of 1, the top path will have an odd number of unmarkable edges, and edge 4 would not be unmarkable. So, there would be an even number of unmarkable edges overall. 

\bigskip
\noindent \textbf{Case D:} The top path will have an even number of unmarkable edges and edge 2 must have a a signal of 1. If edge 5 has a signal of -1, then the bottom path will have an even number of unmarkable edges, and edge 4 is not unmarkable. So, there are an even number of unmarkable edges overall. If edge 5 has a signal of 1, then the bottom path will have an odd number of unmarkable edges, and edge 4 would become unmarkable. So, there would be an even number of unmarkable edges overall. 

\bigskip
\noindent \textbf{Case E:} Edge 2 will have a signal of 1, so the bottom path will have an odd number of unmarkable edges. If edge 6 has a signal of 1, then the top path will have an even number of unmarkable edges and edge 4 would become unmarkable. So, there are an even number of unmarkable edges overall. If edge 5 has a signal of -1, then the top path will have an odd number of unmarkable edges, and edge 4 would not be unmarkable. So, there would be an even number of unmarkable edges overall. 

\end{proof}

\bigskip
\noindent By showing that these cases will always have an even number of unmarkable edges, we prepare ourselves for a winning strategy for Player 1 when there is an odd number of total edges, and a winning strategy for Player 2 when there is an even number of total edges. However, this is not sufficient if both the upper polygon and the lower polygon can make a cycle cell based on our knowledge of the primary edges. Specifically, guaranteeing that a player can force Case A of Figure \ref{fig:guaranteed-even-unmarkable-edges} does not mean they can guarantee a victory. Consider the partially completed game shown in Figure \ref{fig:case-a-fails}. Despite the fact that it is Player 2's turn, there is an even number of total edges, and Player 2 has guaranteed an even number of unmarkable edges with Case A, all legal moves are losing moves. 

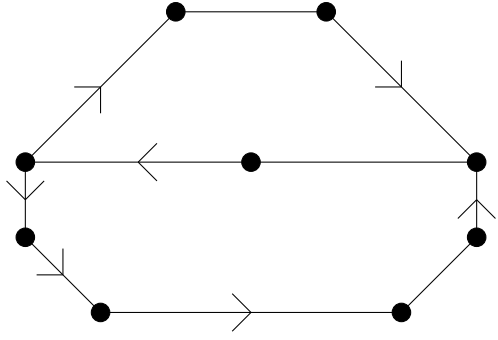
\begin{figure}
    \centering
    \begin{tikzpicture}
    \draw (-3,2)--(3,2);
    \draw (-3,2)--(-1,4);
    \draw (3,2)--(1,4);
    \draw (1,4)--(-1,4);
    \draw (-3,2)--(-3,1);
    \draw (-3,1)--(-2,0);
    \draw (-2,0)--(2,0);
    \draw (2,0)--(3,1);
    \draw (3,1)--(3,2);
    \draw (-2,3)--(-2.35,3);
    \draw (-2,3)--(-2,2.65);
    \draw (-3,1.5)--(-3.25,1.75);
    \draw (-3,1.5)--(-2.75,1.75);
    \draw (-2.5,0.5)--(-2.5,0.85);
    \draw (-2.5,0.5)--(-2.85,0.5);
    \draw (0,0)--(-0.25,0.25);
    \draw (0,0)--(-0.25,-0.25);
    \draw (3,1.5)--(3.25,1.25);
    \draw (3,1.5)--(2.75,1.25);
    \draw (-1.5,2)--(-1.25,2.25);
    \draw (-1.5,2)--(-1.25,1.75);
    \draw (2,3)--(1.65,3);
    \draw (2,3)--(2,3.35);
    \draw [fill=black] (-3.,2.) circle (3.5pt);
    \draw [fill=black] (0.,2.) circle (3.5pt);
    \draw [fill=black] (3.,2.) circle (3.5pt);
    \draw [fill=black] (1.,4.) circle (3.5pt);
    \draw [fill=black] (-1.,4.) circle (3.5pt);
    \draw [fill=black] (-3.,1.) circle (3.5pt);
    \draw [fill=black] (-2.,0.) circle (3.5pt);
    \draw [fill=black] (2.,0.) circle (3.5pt);
    \draw [fill=black] (3.,1.) circle (3.5pt);
    \end{tikzpicture}
    \caption{Despite reaching Case A and there being an even number of edges, Player 2 is in a losing position}
    \label{fig:case-a-fails}
\end{figure}

\bigskip
\noindent Therefore, in order to show that Player 1 will have a winning strategy for $\Theta_{j,2,k}$ graphs with odd edges and Player 2 will have a winning strategy for $\Theta_{j,2,k}$ graphs with even edges, we will need to use Cases B through E, or use case A while proving that only one cycle cell is possible. With these conditions only one of the polygons could potentially become a cycle cell, so by guaranteeing that a player can make the last move, they also guarantee they will be the player to complete the cycle cell.

\bigskip
\noindent What is it about these structures that guarantees an even number of unmarkable edges? It is the fact that all of these structures have an \textit{almost-source} and an \textit{almost-sink} at the two degree-3 vertices. An almost-source is a vertex in which all but one of the edges incident with the vertex are directed outward, and an almost-sink is a vertex in which all but one of the edges incident with the vertex are directed inward. It appears that only one of these properties is true for the four cases B through E. However, the directed inner edge of these four cases actually has an impact on both of the degree-3 vertices, such that having both an almost-sink and almost-source is unavoidable.

\subsection{Winning Strategies}

\noindent Applying Corollary \ref{cor:p-paths-extended}, we have that if at most one of $j$ and $k$ is odd, then Player 1 has a winning strategy when $j + k$ is odd and Player 2 has a winning strategy when $j + k$ is even. If both $j$ and $k$ are odd, this result gives that Player 2 has a winning strategy when both paths have length greater than or equal to 5. In these specific Theta graphs, we are able to deal with these odd short paths.

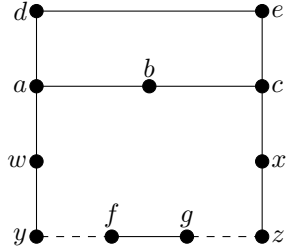
\begin{figure}
    \centering
    \begin{tikzpicture}
        \coordinate (D4) at (0,0);
        \coordinate (D3) at (1,0);
        \coordinate (D2) at (2,0);
        \coordinate (D1) at (3,0);
        \coordinate (C2) at (0,1);
        \coordinate (C1) at (3,1);
        \coordinate (B3) at (0,2);
        \coordinate (B2) at (1.5,2);
        \coordinate (B1) at (3,2);
        \coordinate (A2) at (0,3);
        \coordinate (A1) at (3,3);
        \foreach \i in {A1,A2,B1,B2,B3,C1,C2,D1,D2,D3,D4}
        \filldraw (\i) circle (2.5pt);
        \draw (A1)--(D1);
        \draw (A2)--(D4);
        \draw (D2)--(D3);
        \draw (B1)--(B3);
        \draw (A1)--(A2);
        \draw [dashed] (D1)--(D2);
        \draw [dashed] (D3)--(D4);
        \node[right] at (A1) {$e$};
        \node[left] at (A2) {$d$};
        \node[right] at (B1) {$c$};
        \node[above] at (B2) {$b$};
        \node[left] at (B3) {$a$};
        \node[right] at (C1) {$x$};
        \node[left] at (C2) {$w$};
        \node[right] at (D1) {$z$};
        \node[above] at (D2) {$g$};
        \node[above] at (D3) {$f$};
        \node[left] at (D4) {$y$};
    \end{tikzpicture}
    \caption{Structure for Theorem \ref{thm:3-2-k-case}}
    \label{fig:3-2-k-case}
\end{figure}

For the proof of Theorem \ref{thm:3-2-k-case} refer to Figure \ref{fig:3-2-k-case}.

\begin{theorem}\label{thm:3-2-k-case}
Given a $\Theta_{3,2,k}$ graph with $k \geq 1$ odd Player 2 has a winning strategy.
\end{theorem}

\begin{proof}
\bigskip
\noindent If Player 1 begins by directing an unmirrorable edge, Player 2 should respond by playing the other unmirrorable edge. From here, Player 2 can follow the strategy in the proof of Theorem \ref{thm:p-paths}.

\bigskip
\noindent If $k=1$ and Player 1 begins by directing one the middle edges, without loss of generality consider $a \to b$, then Player 2 should respond by playing $e \to c$. From here it can be exhaustively shown that Player 2 has a winning strategy. Similarly if Player 1 begins by directing a mirrorable primary edge, say $e \to c$, then Player 2 can respond with $a \to b$, and we arrive at the same case.

\bigskip
\noindent If $k > 1$ and Player 1 begins by directing one the middle edges, without loss of generality consider $a \to b$, then Player 2 should respond by playing $a \to d$. From here, Player 2 can use their next move to force one of the cases D or E, and therefore they have a winning strategy. Similarly if Player 1 begins by directing a mirrorable primary edge, say $a \to d$, then Player 2 can respond with $a \to b$, and we arrive at the same case.

\bigskip
\noindent If Player 1 plays on a mirrorable secondary edge, we know that $k \geq 5$. Without loss of generality, we may consider the move $w \to y$. Then, Player 2 should respond with the move $a \to b$. From here, Player 2 can use their next move to direct $x \to c$ or $e \to c$. Since $w \to y$ affects the direction of $\overline{aw}$, we know that this guarantees Case D or E.

\bigskip
\noindent If Player 1 plays on some other edge, we know that $k \geq 7$. From here, Player 2 should play the mirror-reverse move of Player 1's move. Player 2 should use this strategy until Player 1 plays on some edge from one of the other cases. From here, Player 2 should play the strategy given by that case.
\end{proof}

\begin{figure}
    \centering
    \begin{tikzpicture}
    \draw (-1,0)--(1,0);
    \draw (-1,0)--(-1,-1);
    \draw (1,0)--(1,-1);
    \draw (1,0) arc (0:180:1);
    \draw [dashed] (-1,-1)--(1,-1);
    \draw [fill=black] (0,0) circle (2.5pt);
    \draw [fill=black] (-1,0) circle (2.5pt);
    \draw [fill=black] (1,0) circle (2.5pt);
    \draw [fill=black] (-1,-0.5) circle (2.5pt);
    \draw [fill=black] (-1,-1) circle (2.5pt);
    \draw [fill=black] (1,-0.5) circle (2.5pt);
    \draw [fill=black] (1,-1) circle (2.5pt);
    \draw (-1.3,0) node[] {$a$};
    \draw (0,0.3) node[] {$b$};
    \draw (1.3,0) node[] {$c$};
    \draw (-1.3,-0.5) node[] {$d$};
    \draw (1.3,-0.5) node[] {$e$};
    \draw (-1.3,-1) node[] {$f$};
    \draw (1.3,-1) node[] {$g$};
    \end{tikzpicture}
    \caption{Structure for Theorem \ref{thm:1-2-k-case}}
    \label{fig:1-2-k-case}
\end{figure}
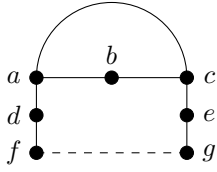

\begin{theorem}\label{thm:1-2-k-case}
Given a $\Theta_{1,2,k}$ graph with $k \geq 5$ odd Player 2 has a winning strategy.
\end{theorem}

\begin{proof}
\bigskip
\noindent If Player 1 begins by directing edge $\overline{ac}$. Then, Player 2 should direct either of $\overline{ab}$ or $\overline{bc}$ so that the 3-cycle cannot be completed. From this point, we have reached one of cases B-E of Figure \ref{fig:guaranteed-even-unmarkable-edges}, so there will be an even number of unmarkable edges and at most one cycle cell, so Player 2 has a winning strategy.

\bigskip
\noindent If Player 1 begins by directing one of $\overline{ab}$ or $\overline{bc}$, then Player 2 should direct $\overline{ac}$ so the triangle cycle cannot be completed. This is then identical to the case above.

\bigskip
\noindent If Player 1 begins by directing one of $\overline{ad}$ or $\overline{df}$, then direct $\overline{bc}$ so that both edges are directed toward vertex $a$ or both edges are directed away from vertex $a$. From here, Player 2 may force one of the cases B-E by directing $\overline{ac}$ accordingly on their next move, or completing a cycle cell if it is possible.

\bigskip
\noindent If Player 1 begins by directing one of $\overline{ce}$ or $\overline{eg}$, then direct $\overline{ab}$ so that both edges are directed toward vertex $c$ or both edges are directed away from vertex $c$. From here, Player 2 may force one of the cases B-E by directing $\overline{ac}$ accordingly on their next move, or completing a cycle cell if it is possible.

\bigskip
\noindent If Player 1 begins by directing any other edge, Player 2 should direct edge $\overline{ac}$. From here, Player 2 may force one of the cases B-E by directing one of $\overline{ab}$ or $\overline{bc}$ accordingly on their next move, or completing a cycle cell if it is possible.
\end{proof}

Combining the results from Corollary \ref{cor:p-paths-extended} and Theorems \ref{thm:3-2-k-case} and \ref{thm:1-2-k-case}, we arrive at the following result.

\begin{corollary}\label{cor:j-2-k-full}
Consider a $\Theta_{j,2,k}$-graph with $j \geq 1$ and $k \geq 2$. If $j+k$ is odd, Player 1 has a winning strategy. If $j+k$ is even, Player 2 has a winning strategy.
\end{corollary}

Another way of stating this theorem is that the parity of edges determines the winner for any simple $\Theta_{j,2,k}-$graph.

\bigskip
\noindent Note that this result answers the question asked by Alvarado at al. in  \cite{alvarado2021game} about the gameboard of a heptagon and pentagon sharing two consecutive edges that motivated our research.

\section{Rectangular Grids}

In \cite{alvarado2021game}, Theorem 6 states that given a graph with an involution that leaves each cell nowhere-involutive or self-involutive, Player 1 has a winning strategy when there is exactly one edge fixed by the involution, and Player 2 has a winning strategy when no edge is fixed by the involution.

\bigskip
\noindent Consider a graph that is a grid of $b \times a$ vertices. That is, the height of the grid is $b-1$ edges and the width of the grid is $a-1$ edges. For such a graph, we are interested in the involution of 180 degree rotation since any other involution would fix multiple edges.

\bigskip
\noindent Based on Theorem 6 in \cite{alvarado2021game}, Player 2 will win when $a$ and $b$ have the same parity. However, when $a$ and $b$ have opposite parity, the theorem does not apply. In this case there is exactly one self-involutive edge, but this makes two cells part-involutive. We are going to settle this case below.

\bigskip
\noindent We consider the grid game board to be the first quadrant with vertices placed at the integer coordinates of the coordinate plane with the leftmost path as the $y$-axis and the bottom most path as the $x$-axis. Then the left bottom vertex may be labelled $(0,0)$, and the top right vertex may be labelled $(a,b)$. We will identify a vertex with its label. In the involution of rotating this graph by 180 degrees the center of rotation is the midpoint of an edge, and the image of vertex $(i,j)$ is $(a-i,b-j)$. 

\bigskip
\noindent Player $W$ will use the rotate-reverse strategy. This strategy involves rotating the other player's most recent move 180 degrees and then reversing the direction. Consider the move $(i_1, j_1) \rightarrow (i_2, j_2)$. Rotating this move 180 degrees gives the move $(a - i_1, b - j_1) \rightarrow (a - i_2, b - j_2)$. Then, reversing this move gives $(a - i_2, b - j_2) \rightarrow (a - i_1, b - j_1)$.

\bigskip
\noindent We will denote the image of vertex $v$ by $v^{\prime}$. In Figure \ref{fig:rotate-reverse-move}, the bottom left vertex is $(0,0)$ and the top right vertex is $(3,2)$. The move $x \rightarrow y$ has the rotate-reverse move $y^{\prime} \rightarrow x^{\prime}$. That is, the rotate-reverse of the move $(0,0) \rightarrow (1,0)$ is $(2,2) \rightarrow (3,2)$.

\begin{figure}
    \centering
    \begin{tikzpicture}
    \draw [fill=black] (0,0) circle (2.5pt);
    \draw [fill=black] (0,1) circle (2.5pt);
    \draw [fill=black] (0,2) circle (2.5pt);
    \draw [fill=black] (1,0) circle (2.5pt);
    \draw [fill=black] (1,1) circle (2.5pt);
    \draw [fill=black] (1,2) circle (2.5pt);
    \draw [fill=black] (2,0) circle (2.5pt);
    \draw [fill=black] (2,1) circle (2.5pt);
    \draw [fill=black] (2,2) circle (2.5pt);
    \draw [fill=black] (3,0) circle (2.5pt);
    \draw [fill=black] (3,1) circle (2.5pt);
    \draw [fill=black] (3,2) circle (2.5pt);
    \draw (0,0)--(0,2);
    \draw (1,0)--(1,2);
    \draw (2,0)--(2,2);
    \draw (3,0)--(3,2);
    \draw (0,0)--(3,0);
    \draw (0,1)--(3,1);
    \draw (0,2)--(3,2);
    \draw (0,-0.3) node[] {$x$};
    \draw (1,-0.3) node[] {$y$};
    \draw (3,2.3) node[] {$x^{\prime}$};
    \draw (2,2.3) node[] {$y^{\prime}$};
    \draw (0.5,0)--(0.3,0.2);
    \draw (0.5,0)--(0.3,-0.2);
    \draw (2.5,2)--(2.3,2.2);
    \draw (2.5,2)--(2.3,1.8);
    \end{tikzpicture}
    \caption{Example of Rotate-Reverse Move}
    \label{fig:rotate-reverse-move}
\end{figure}
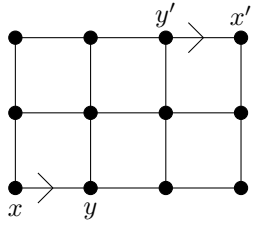

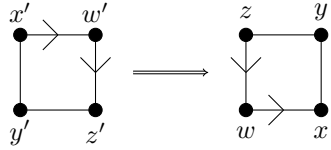
\begin{figure}
    \centering
    \begin{tikzpicture}
    \draw [fill=black] (0,0) circle (2.5pt);
    \draw [fill=black] (0,1) circle (2.5pt);
    \draw [fill=black] (1,0) circle (2.5pt);
    \draw [fill=black] (1,1) circle (2.5pt);
    \draw (0,-0.3) node[] {$y^{\prime}$};
    \draw (1,-0.3) node[] {$z^{\prime}$};
    \draw (0,1.3) node[] {$x^{\prime}$};
    \draw (1,1.3) node[] {$w^{\prime}$};
    \draw (0,0)--(1,0);
    \draw (1,0)--(1,1);
    \draw (1,1)--(0,1);
    \draw (0,1)--(0,0);
    \draw (0.5,1)--(0.3,1.2);
    \draw (0.5,1)--(0.3,0.8);
    \draw (1,0.5)--(1.2,0.7);
    \draw (1,0.5)--(0.8,0.7);
    \draw [-{implies},double] (1.5,0.5) -- (2.5,0.5);
    \draw [fill=black] (3,0) circle (2.5pt);
    \draw [fill=black] (4,0) circle (2.5pt);
    \draw [fill=black] (3,1) circle (2.5pt);
    \draw [fill=black] (4,1) circle (2.5pt);
    \draw (3,-0.3) node[] {$w$};
    \draw (4,-0.3) node[] {$x$};
    \draw (3,1.3) node[] {$z$};
    \draw (4,1.3) node[] {$y$};
    \draw (3,0)--(4,0);
    \draw (4,0)--(4,1);
    \draw (4,1)--(3,1);
    \draw (3,1)--(3,0);
    \draw (3.5,0)--(3.3,0.2);
    \draw (3.5,0)--(3.3,-0.2);
    \draw (3,0.5)--(3.2,0.7);
    \draw (3,0.5)--(2.8,0.7);
    \end{tikzpicture}
    \caption{Player 1 will only make a cycle available if Player 2 makes a cycle available}
    \label{fig:grid-cycle-image}
\end{figure}

\begin{theorem}\label{thm:grids}
Consider a graph that is a grid of $a \times b$ vertices, where exactly one of $a$ or $b$ is odd. Player 1 has a winning strategy for this gameboard.
\end{theorem}

\begin{proof}

Because one of the dimensions of the grid is odd, there is exactly one self-involutive edge under the involution of 180 degree rotation. Player 1 starts by playing this self-involutive edge. That is, if the move is $u \rightarrow v$, then $u = v^{\prime}$ and $v = u^{\prime}$.

\bigskip
\noindent From here, whenever Player 2 plays $x \rightarrow y$, Player 1 should respond with $x^{\prime} \rightarrow y^{\prime}$, unless it is possible to complete a cycle, in which case Player 1 will do that instead.

\bigskip
\noindent Assume that up until a certain point in the game, Player 1 has used this strategy. Then, we will show that Player 1's move must be available, legal, and it will not create a winning move for Player 2.

\bigskip
\noindent Player 1's move must be available. If the move $y^{\prime} \rightarrow x^{\prime}$ had already been played, then there are two cases. Either Player 1 played it in response to $x \rightarrow y$ having been played earlier, in which case, Player 2 could not have played $x \rightarrow y$ on the most recent move. Or, if Player 2 played $y^{\prime} \rightarrow x^{\prime}$, then Player 1 would have played $x \rightarrow y$ in response, which means Player 2 couldn't have played $x \rightarrow y$ on the most recent move.

\bigskip
\noindent Player 1's move must be legal, as if the move $y^{\prime} \rightarrow x^{\prime}$ creates a sink or a source, then the previous move $x \rightarrow y$ by Player 2 must have created a sink or a source as well. 

\bigskip
\noindent Consider a fixed $y$ vertex. Then, whenever Player 2 plays $x \rightarrow y$ and Player 1 plays $y^{\prime} \rightarrow x^{\prime}$ in response, there is an edge directed towards $y$ and an edge directed away from $y^{\prime}$. Therefore, Player 1 can only create a source by playing $y^{\prime} \rightarrow x^{\prime}$ if Player 2 has created a sink by playing $x \rightarrow y$. 

\bigskip
\noindent Therefore, Player 1's move must be legal. Since Player 1's move must always be legal and available, Player 1 will always make the last move if there are no cycle cells made. We next show that Player 1 will always have a winning move available before Player 2 does.

\bigskip
\noindent Player 1's move will not create a winning move for Player 2, because if the rotate-reverse move would create a winning move for Player 2, then Player 2's most recent move has created a winning move for Player 1 which they will make instead. Consider Figure \ref{fig:grid-cycle-image} with vertices $w,x,y,$ and $z$, and its 180 degree rotation with the vertices $w^{\prime}, x^{\prime}, y^{\prime},$ and $z^{\prime}$.

\bigskip
\noindent Without loss of generality, assume the rotate-reverse move for Player 1 to make is $y^{\prime} \rightarrow x^{\prime}$ and this would give Player 2 a winning move. However, this means that Player 2 has just played $x \rightarrow y$ which created a winning move for Player 1, namely the move $y \rightarrow z$. Player 1 will make this instead $y^{\prime} \rightarrow x^{\prime}$. 

\end{proof}

\section{Conclusion}

It would be possible to expand the class of $\Theta_{j,2,k}$-graphs to $\Theta_{j,n,k}$-graphs. Although most of these graphs are covered by Theorem \ref{thm:p-paths}, not all. While expanding Theorem \ref{thm:p-paths} directly would be more efficient, finding solutions for a generalized $\Theta_{j,n,k}$ would be easier, and would finish the class of Theta graphs. To complete this, one would need to find strategies for:

\begin{itemize}
    \item $n = 3$
    \begin{itemize}
        \item $\Theta_{1,3,k}$ with $k \geq 2$
        \item $\Theta_{3,3,k}$ with $k \geq 2$
        \item $\Theta_{j,3,k}$ with $j \geq 5$ odd and $k \geq 2$ even
    \end{itemize}
    \item $n \geq 4$ even
    \begin{itemize}
        \item $\Theta_{1,n,k}$ with $k \geq 3$ odd
        \item $\Theta_{3,n,k}$ with $k \geq 3$ odd
    \end{itemize}
    \item $n \geq 5$ odd 
    \begin{itemize}
        \item $\Theta_{1,n,k}$ with $k \geq 2$ even
        \item $\Theta_{3,n,k}$ with $k \geq 2$ even
    \end{itemize}
\end{itemize}

\bigskip
\noindent For all the game boards we examined the parity of the total number of edges determined the winner of the game. The parity conjecture made in \cite{alvarado2021game} states that Player 1 has a winning strategy when there is an odd number of total edges and Player 2 has a winning strategy when there is an even number of total edges. However, consider Figure \ref{fig:counterexample}. In this graph, there are four edges, but by playing $\overrightarrow{cd}$ to begin, Player 1 has won the game. This is because $\overline{ab}$ is automatically an unmarkable edge, since it must create either a sink or a source at vertex $a$. So, we modify the parity conjecture as follows:

\begin{conjecture}
For game boards in which there is no vertex with degree 1, Player 1 has a winning strategy when there is an odd number of total edges, and Player 2 has a winning strategy when there is an even number of total edges.
\end{conjecture}

Alternatively,

\begin{conjecture}
If we do not consider a degree 1 vertex either a sink or a source, Player 1 has a winning strategy when there is an odd number of total edges, and Player 2 has a winning strategy when there is an even number of total edges.
\end{conjecture}

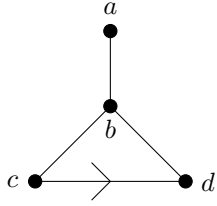
\begin{figure}
    \centering
    \begin{tikzpicture}
    \draw [fill=black] (0,0) circle (2.5pt);
    \draw [fill=black] (-1,-1) circle (2.5pt);
    \draw [fill=black] (1,-1) circle (2.5pt);
    \draw [fill=black] (0,1) circle (2.5pt);
    \draw (0,-0.3) node[] {$b$};
    \draw (-1.3,-1) node[] {$c$};
    \draw (0,1.3) node[] {$a$};
    \draw (1.3,-1) node[] {$d$};
    \draw (0,0)--(1,-1);
    \draw (0,0)--(-1,-1);
    \draw (1,-1)--(-1,-1);
    \draw (0,0)--(0,1);
    \draw (0,-1)--(-0.25,-0.75);
    \draw (0,-1)--(-0.25,-1.25);
    \end{tikzpicture}
    \caption{Player 1 has a winning strategy from this position}
    \label{fig:counterexample}
\end{figure}

There are also new versions of the game that could be played. The most simple of these would to ignore the geometric drawing of the graph, and to make the object of the game to complete a cycle (as defined in graph theory), rather than a cycle cell (from topology). In fact, how much the geometric drawing of a gameboard can affect game play is a nontrivial question. Consider Figure \ref{fig:geometry-issue}. These two graphs are isomorphic. However, because of the cycle cell rules in the Game of Cycles, they must be played differently. The image on the left has two cells of length four and two of length three, while the image on the right has three cells of length four and one of length three. Knowing when a different drawing of the planar graph affects the game could be important to understanding the Game of Cycles more generally.

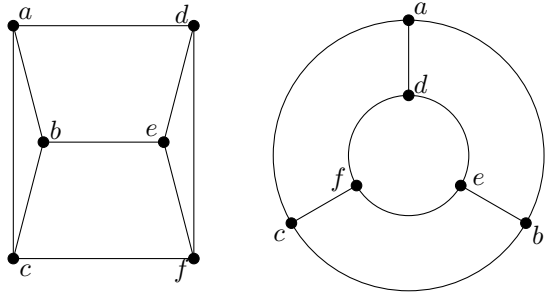
\begin{figure}
    \centering
    \begin{tikzpicture}[scale = 0.8]
    \draw (0,0)--(2,0);
    \draw (0,0)--(-0.5,1.9365);
    \draw (0,0)--(-0.5,-1.9365);
    \draw (2,0)--(2.5,1.9365);
    \draw (2,0)--(2.5,-1.9365);
    \draw (-0.5,1.9365)--(2.5,1.9365);
    \draw (-0.5,-1.9365)--(2.5,-1.9365);
    \draw (-0.5,1.9365)--(-0.5,-1.9365);
    \draw (2.5,1.9365)--(2.5,-1.9365);
    \draw [fill=black] (0,0) circle (2.5pt);
    \draw [fill=black] (2,0) circle (2.5pt);
    \draw [fill=black] (-0.5,1.9365) circle (2.5pt);
    \draw [fill=black] (-0.5,-1.9365) circle (2.5pt);
    \draw [fill=black] (2.5,1.9365) circle (2.5pt);
    \draw [fill=black] (2.5,-1.9365) circle (2.5pt);
    \draw (-0.3,2.1365) node[] {$a$};
    \draw (0.2,0.2) node[] {$b$};
    \draw (-0.3,-2.1365) node[] {$c$};
    \draw (2.3,2.1365) node[] {$d$};
    \draw (1.8,0.2) node[] {$e$};
    \draw (2.3,-2.1365) node[] {$f$};
    \end{tikzpicture}
    \hspace{0.25in}
    \begin{tikzpicture}[scale = 0.8]
`   \draw (0,0) circle (2.25cm);
    \draw (0,0) circle (1cm);
    \draw (0,1)--(0,2.25);
    \draw (-.866,-0.5)--(-1.9486,-1.125);
    \draw (.866,-0.5)--(1.9486,-1.125);
    \draw [fill=black] (0,1) circle (2.5pt);
    \draw [fill=black] (0,2.25) circle (2.5pt);
    \draw [fill=black] (-.866,-0.5) circle (2.5pt);
    \draw [fill=black] (-1.9486,-1.125) circle (2.5pt);
    \draw [fill=black] (.866,-0.5) circle (2.5pt);
    \draw [fill=black] (1.9486,-1.125) circle (2.5pt);
    \draw (0.2,2.45) node[] {$a$};
    \draw (2.1486,-1.325) node[] {$b$};
    \draw (-2.1486,-1.325) node[] {$c$};
    \draw (0.2,1.2) node[] {$d$};
    \draw (1.166,-0.4) node[] {$e$};
    \draw (-1.166,-0.4) node[] {$f$};
    \end{tikzpicture}
    \caption{These are the same graph, but they have a different cell type}
    \label{fig:geometry-issue}
\end{figure}

\bigskip
\noindent One other step would be to play the Game of Cycles on a tree, thus rendering the ``cycle" aspect worthless. In fact, Mathews \cite{mathews2021game} has already begun to consider this topic, naming it the Game of Arrows. Specifically, they have focused on 3-legged spider graphs. Here, Mathews makes the rule that leaves of the graph can be directed, allowing for some sources and sinks. The work in this paper regarding signals and interactions may be effective for this topic.

\bigskip
\noindent Lastly, it could be interesting to play the game on nonplanar graphs embedded in different surfaces.

\bibliographystyle{plain}
\bibliography{bibliography.bib}

\end{document}